%% file: main.tex
\documentclass[reqno]{amsart}
\pdfoutput=1

\usepackage[utf8]{inputenc}
\usepackage[english]{babel}
\usepackage{amssymb}
\usepackage{hyperref}
\usepackage{enumitem}
\usepackage{etoolbox}
\usepackage{braket} 
\usepackage{faktor}
\usepackage{mathtools}
\usepackage{stackrel}
\usepackage{verbatim}
\usepackage{tikz-cd}
\usepackage[english]{cleveref}
\usepackage{todonotes}
\usepackage{ytableau}

\theoremstyle{definition}
\newtheorem{theorem}{Theorem}[section]

\newtheorem*{theo*}{Theorem}
\newtheorem{definition}[theorem]{Definition}
\newtheorem{remark}[theorem]{Remark}
\newtheorem{example}[theorem]{Example}
\newtheorem{lemma}[theorem]{Lemma}
\newtheorem{proposition}[theorem]{Proposition}
\newtheorem{corollary}[theorem]{Corollary}
\newtheorem{conjecture}[theorem]{Conjecture}

\newcommand{\Z}{\mathbb{Z}}
\newcommand{\N}{\mathbb{N}}
\newcommand{\Q}{\mathbb{Q}}
\newcommand{\C}{\mathbb{C}}

\newcommand{\PP}{\mathbb{P}} 
\newcommand{\F}{\mathrm{F}} 
\newcommand{\SG}{\mathfrak{S}} 
\newcommand{\SL}{SL_2(\Q)} 
\newcommand{\size}{l} 

\DeclareMathOperator{\Aut}{Aut} 
\DeclareMathOperator{\conf}{Conf}
\DeclareMathOperator{\Ext}{Ext} 

\DeclareMathOperator{\gr}{gr}
\DeclareMathOperator{\rk}{rk}

\DeclareMathOperator{\Hom}{Hom} 
\DeclareMathOperator{\Ind}{Ind} 
\DeclareMathOperator{\im}{Im} 
\DeclareMathOperator{\dd}{d} 

\makeatletter
\newcommand{\mylabel}[2]{#2\def\@currentlabel{#2}\label{#1}}
\makeatother

\makeatletter
\newcommand*{\bigcdot}{%
  {\mathbin{\mathpalette\bigcdot@{}}}%
}
\newcommand*{\bigcdot@scalefactor}{.75}
\newcommand*{\bigcdot@widthfactor}{1.4}
\newcommand*{\bigcdot@}[2]{%
  \sbox0{$#1\vcenter{}$}
  \sbox2{$#1\cdot\m@th$}%
  \hbox to \bigcdot@widthfactor\wd2{%
    \hfil
    \raise\ht0\hbox{%
      \scalebox{\bigcdot@scalefactor}{%
        \lower\ht0\hbox{$#1\bullet\m@th$}%
      }%
    }%
    \hfil
  }%
}
\makeatother

\robustify{\bigcdot}

\allowdisplaybreaks

\makeatletter
\g@addto@macro{\UrlBreaks}{\UrlOrds}
\g@addto@macro{\UrlBreaks}{%
\do\/\do\d%
}
\makeatother

\begin{document}

\title[Ordered configuration spaces on an elliptic curve]{Asymptotic growth of Betti numbers of ordered configuration spaces on an elliptic curve}

\author[R. Pagaria]{Roberto Pagaria}
\thanks{The author is supported by PRIN 2017YRA3LK}
\address{Roberto Pagaria \newline \textup{Università di Bologna, Dipartimento di Matematica}\\ Piazza di Porta San Donato 5 - 40126 Bologna\\ Italy.}
\email{roberto.pagaria@unibo.it}

\begin{abstract}
We construct a dga to computing the cohomology of ordered configuration spaces on an algebraic variety with vanishing Euler characteristic. 
It follows that the $k$-th Betti number of $\conf(C,n)$ ($C$ is an elliptic curve) grows as a polynomial of degree exactly $2k-2$.
We also compute $H^k(\conf(C,n))$ for $k\leq 5$ and arbitrary $n$.
\end{abstract}

\maketitle

The ordered configuration space of $n$ points on a smooth projective variety $X$ is
\[ \conf(X,n) = \{ (p_1, \dots, p_n) \in X^n \mid p_i \neq p_j\}.\]
A central problem in the theory of configuration spaces is understand the cohomology of these topological spaces.
The main tool is the Kri\v{z} model, introduced at the same time by Kri\v{z} \cite{Kriz94} and by Totaro \cite{Totaro}, it is a differential graded algebra $E(X,n)$ that codifies the rational homotopy type of $\conf(X,n)$ (see \Cref{sect:1}).
Often is useful to study all these configuration spaces $\conf(X,n)$ for $n\in \N$ all together.
In the case of vanishing Euler characteristic $\chi(X)=0$, we construct a filtration $F_\bigcdot$ of the Kri\v{z} model $E(X,n)$ and we prove that the differential is strict with respect to this filtration (\Cref{thm:grad_cohomology}).
Passing to the graded model we obtain a simpler differential graded algebra that is more feasible for computing cohomology.
The mixed Hodge numbers of $\conf(X,n)$ (indeed any polynomial function $\N \to \N$) can be written as linear combination of binomials with coefficient in $\Z$.
We prove in \Cref{prop:coeff_binom} that these coefficients are nonnegative integers for any algebraic variety with $\chi(X)=0$, this is not true in a wider generality. 

We apply the simpler dga to configuration on an elliptic curve $C$.
It was known from \cite{Church12} that the $k$-th Betti number of $\conf(C,n)$ grows as polynomial of degree at most $2k$.
We improve that result by showing that the Betti numbers grow as a polynomial of degree exactly $2k-2$ (\Cref{cor:Betti}).
We developed a method to give lower bounds for the mixed Hodge numbers (and so for Betti numbers) involving certain partitions, see \Cref{lemma:oyster}.
The lower bounds are extremely good for understand the asymptotic behaviour of these numbers.

The upper bound for these number is more difficult to determine, so we developed an algorithm for computing the (not so) small cases. The algorithm uses in a essential way the aforementioned filtration of the Kri\v{z} model.
We finally present the mixed Hodge numbers of $\conf(C,n)$ for $n\leq 7$ (\Cref{appendix}).
Using the theory development here and the computation of small cases, we also determine the $\dim H^k(\conf(C,n);\Q)$ for $k\leq 5$ and arbitrary $n\in \N$ (\Cref{cor:Betti}).

\medskip
A recent tool to deal with sequence of homogeneous objects is the representation stability introduced by Church and Farb \cite{CF13}.
We use a variant of that: we substitute the category of finite sets with injections with the category $\F$ of finite sets with all maps.
The representation theory of $\F$ is developed by \cite{WG14}, \cite{SS17}, and \cite{Ryba}. 
Ellenberg and Wiltshire-Gordon \cite{EWG15} have shown that the spaces $\conf(X,n)$ are an $\F$-module if $X$ has two linearly independent vector fields.
Their result implies that the cohomology is an $\F$-module, we give a short proof that the Kri\v{z} model is an $\F$-module.
We use this structure in an essential way to define the filtration $F$ and to prove that is strictly compatible with the differential.

\section{Representation theory of the Kri\v{z} model}
\label{sect:1}
\subsection*{The Kri\v{z} model}
Let $X$ be a smooth projective variety.
For each element $x\in H^\bigcdot(X)$ we denote by $x_i \in H^\bigcdot (X^n)$ its image under the map $H^\bigcdot (X) \hookrightarrow H^\bigcdot (X^n)$ induced by the projection $X^n \to X$ on the $i$-th factor.

The class of the diagonal $\Delta \in H^{2\dim_\C X}(X \times X)$ is the cohomological class of the subvariety $\{(x,x) \mid x \in X\}$ in $X \times X$.
If we fix a graded basis $\{b_1, \dots, b_k\}$ of $H^\bigcdot(X)$ and consider the dual basis $\{b_1^*, \dots, b_k^*\}$ with respect to the cup product, then
\begin{equation} \label{eq:diag_formula}
\Delta = \sum_{j=1}^k (-1)^{\deg b_j^*} b_j \otimes b_j^*
\end{equation}
Consider also $\Delta_{i,j} \in H^\bigcdot (X^n)$ as the pullback of the diagonal $\Delta$ with respect to the projection $X^n \to X^2$ on the $i$-th and $j$-th factors.

\begin{definition}
Let $E(X,n)$ be the differential bigraded algebra 
\[ \faktor{H\bigcdot(X^n)[\{G_{i,j}\}_{i<j}]}{I},\]
where $\deg(x_i)=(\deg x, 0)$, $\deg(G_{i,j})=(0, 2\dim_\C X -1)$, and $I$ is the ideal generated  by
\begin{align*}
&G_{i,j}(x_i -x_j) &&\textnormal{for } i<j \textnormal{ and } x \in H^\bigcdot(X) \\
& G_{i,j}G_{j,k}-G_{i,j}G_{i,k}+ G_{j,k}G_{i,k} && \textnormal{for } i<j<k.
\end{align*}
The differential is defined by $\dd(x_i)=0$ for all $x$ and $i$ and $\dd(G_{i,j})=\Delta_{i,j}$.
\end{definition}

For the sake of notation, we define for $i>j$ $G_{i,j}=G_{j,i}$.
The following theorem was proven by Kri\v{z} \cite{Kriz94} and Totaro \cite{Totaro} independently

\begin{theorem}
The dga $E(X,n)$ is a rational model for $\conf(X,n)$.
Therefore $H^\bigcdot (\conf(X,n)) \cong H^{\bigcdot} (E(X,n),\dd)$.
\end{theorem}

\subsection*{The \texorpdfstring{$\F$}{F} module structure}
We will consider the spaces $\conf(X,n)$ and their cohomology groups for $n \in \N$ all together.
Consider the \textit{category of finite sets} $\F$ whose objects are the sets $[n] = \{1,2, \dots, n\}$ for all $n \in \N$ and with morphisms all the maps of sets between them.
Notice that the automorphism of the object $[n]$ $\Aut_{\F}([n]) \cong \SG_n$ is the symmetric group.


An $\F$-\textit{representation} over $\mathcal{V}$ is a functor $V$ from the category $\F$ to the category $\mathcal{V}$.
A $\F$-\textit{module} $V$ is \textit{finitely generated} if there exists elements $v_i \in V([n_i])$ for $i$ in a finite set $I$ such that no proper submodule of $V$ contains all $v_i$ for $i \in I$.
A morphism $\tau \colon V \to W$ of $\F$-\textit{module} is a natural transformation between the functors $V$ and $W$.

\begin{definition}
The $\F$-module $E(X)$ is the functor defined by $[n] \mapsto E(X,n)$ on the objects and for each map $f\colon [n] \to [m]$ associate the morphism $f_* \colon E(X,n) \to E(X,m)$ defined by
\[ x_i \mapsto x_{f(i)}, \qquad G_{i,j} \mapsto \begin{cases} G_{f(i),f(j)} & \textnormal{if } f(i) \neq f(j) \\
0 & \textnormal{otherwise} 
\end{cases} \]
\end{definition}
The submodules $E^{p,q}(X)$ are finitely generated and $\dd \colon E(X) \to E(X)$ commutes with $f_*$ for all injective maps $f$. 

\begin{proposition}
If the Euler characteristic of $X$ is zero, i.e.\ $\chi(X)=0$. Then $\dd$ is a morphism of $\F$-modules.
\end{proposition}
\begin{proof}
We verify that $f(i)=f(j)$ implies $f_*(\dd(G_{i,j}))=0$, the other conditions are trivial.
Let $\{b_1, \dots, b_k\}$ be a graded basis of $H^\bigcdot(X)$ and $\{b_1^*, \dots, b_k^*\}$ be the dual basis, we use \Cref{eq:diag_formula}:
\begin{align*}
f_*(\dd(G_{i,j})) &= f_* \left( \sum_{l=1}^k (-1)^{\deg b_l^*} (b_l)_i  (b_l^*)_j \right) \\
&= \sum_{l=1}^k (-1)^{\deg b_l^*} (b_l)_{f(i)}  (b_l^*)_{f(i)} \\
&= \sum_{l=1}^k (-1)^{\deg b_l^*} [X]_{f(i)} \\
&= \sum_{r=1}^{2\dim(X)} (-1)^r h_r(X) [X]_{f(i)}\\
&= \chi(X) [X]_{f(i)},
\end{align*}
where $h_r(X)$ is the $r$-th Betti number of $X$.
The hypothesis $\chi(X)=0$ completes the proof.
\end{proof}

A topological interpretation of the above result can be found in \cite{EWG15}.

\subsection*{Representation theory of \texorpdfstring{$\F$}{F}}
The irreducible representations of the symmetric group $\SG_n$ are parametrize by the partitions $\lambda \vdash n$ of the number $n$, we denote them by $V_\lambda$.
The category $\mathcal{V}^\F$ of the $\F$ representation over $\mathcal{V}$ is not semisimple, however has the Jordan H\"{o}lder property as shown in \cite[Corollary 5.4]{WG14}.
The \textit{Schur projective} representation of type $\lambda \vdash k$ is the functor
\[ \PP_\lambda ([n])= \mathbb{S}_\lambda(\Q^n),\]
where $\mathbb{S}_\lambda$ is the Schur functor (see \cite{Weyman,Macdonald} for an introductory exposition of Schur functors).
The dimension $\dim \PP_\lambda([n])$ is an evaluation of the Schur symmetric polynomial, i.e.\ $s_\lambda(1^n)$.
These projective representations have the following property
\begin{equation}
\label{eq:scalar_P}
\Hom_{\mathcal{V}^\F}(\PP_\lambda, V) \cong \Hom_{\SG_k}(V_\lambda, V([k])),
\end{equation}
where $\lambda \vdash k$.

The simple representations are classified by Wiltshire-Gordon and they are of two kinds: $D_k$ for $k \in \N$ and $C_\lambda$ for $\lambda \vdash k$ with $\lambda_1>1$.
The representation $C_\lambda$ for $\lambda \vdash k$ is defined by $C_\lambda([n])=0$ for $n<k$ and by
\[ C_\lambda([n])= \Ind_{\SG_k \times \SG_{n-k}}^{\SG_n} V_\lambda \boxtimes 1_{n-k},\]
where $1_{n}$ is the trivial representation of $\SG_n$ and the dimension is $\binom{n}{k} \langle s_\lambda, p_{1^k}\rangle$.
The representation $D_0$ is defined by $D_0([0])=\Q$ and $D_0([n])=0$ for all $n>0$.
Finally, $D_k$ is defined by $D_k([n])=0$ for $n<k$ and by
\[ D_k([n])= V_{(n-k+1, 1^{k-1})}. \]
Moreover $D_k([n])$ has dimension $\binom{n-1}{k-1}$ for $n,k>0$.

\begin{definition}
A finitely generated representation $V$ is of degree $k$ if there exists a surjection $\oplus_i \PP_{\lambda_i} \twoheadrightarrow V$ with $\lambda_i \vdash n_i$ and $n_i\leq k$.
\end{definition}

Since there exist surjections $\PP_\lambda \twoheadrightarrow C_\lambda$ and $\PP_{1^k} \twoheadrightarrow D_k$, $D_k$ has degree $k$ and $C_\lambda$ has degree $|  \lambda |$.

\begin{proposition}[{\cite[Theorem 6.39]{WG14}} ] \label{prop:resolutions}
Every finitely generated representation $V$ of degree $k$ has a finite resolution 
\[ 0 \to \PP_0 \oplus \mathbb{D}_0 \to  \dots \to \PP_{k-1} \oplus \mathbb{D}_{k-1} \to \PP_k \to V \to 0,\]
where $\PP_i$ is a direct sum of Schur projective of pure degree $\leq i$ and $\mathbb{D}_i$ is direct sum of some copies of $D_{i+1}$ and $D_{i+2}$.
Moreover, $D_k$ has an infinite projective resolution given by
\[ \dots \to \PP_{1^{k+2}} \to \PP_{1^{k+1}} \to \PP_{1^{k}} \to D_k \to 0.\]
\end{proposition}

\begin{lemma} \label{lemma:ext}
We have the following:
\begin{enumerate}
\item if $h \neq k+i$ then $\Ext^i (D_k,D_h)=0$, and $\Ext^i (D_k,D_{k+i})=\Q$,
\item let $V$ a representation of degree $k$, then $\Ext^i(V, C_\lambda)=0$ if $k< |\lambda | +i$.
\end{enumerate} 
\end{lemma}
\begin{proof}$ $
\begin{enumerate}
\item Consider the resolution $\PP_\bigcdot$ of $D_k$ from \Cref{prop:resolutions}, we compute the $\Ext$ functor using the projective resolution 
$
\Ext^i(D_k,D_h) = H^i(\Hom(\PP_{\bigcdot},D_h)).
$
The complex $\Hom(\PP_{\bigcdot},D_h)$ is almost trivial because 
\[ \Hom_{\F}(\PP_{1^{k+i}},D_h) \cong \Hom_{\SG_{k+i}}(V_{1^{k+i}}, V_{(k+i-h+1,1^{h-1})}) \]
is zero for $k+i \neq h$ and for $h=k+i$ is $\Q$ by the Schur Lemma.
\item Consider the resolution $\PP_\bigcdot \oplus \mathbb{D}_\bigcdot$ of $V$ from \Cref{prop:resolutions}.
It is enough to verify that $\Hom_\F (\PP_{k-i} \oplus \mathbb{D}_{k-i}, C_\lambda)=0$, indeed $\Hom(D_a, C_\lambda)=0$ because they are two different simple representations and 
\[ \Hom_{\F}(\PP_\mu, C_\lambda)= \Hom_{\SG_{|\mu |}}(V_{\mu},0)=0,\]
since $|\mu | \leq k-i < | \lambda |$. \qedhere
\end{enumerate}
\end{proof}

Let $V$ be a finitely generated representation of $\F$ of degree $k$.
Define a canonical filtration $F_\bigcdot$ such that $F_i\subseteq V$ is the submodule generated by $V([j])$ for all $j\leq i$.


\begin{example}
Consider the short exact sequence
\[ 0 \to D_2 \to \PP_{1} \to D_1 \to 0,\]
on the object $[n]$, $\PP_{1}([n])=\Q^n$, $D_2([n])\cong \Q^{n-1}$ is the subobject of vectors whose coordinates have zero sum, and $D_1([n]) \cong \Q$ is their quotient.
The sequence does not split and the degree of the subject $D_2$ (equals to $2$) is bigger then the degree of $\PP_1$ (equals to $1$).

The inclusion $j\colon D_2([n]) \hookrightarrow \PP_1([n])$ has rank $n-1$, but the graded map $\gr_F j$ is zero.
\end{example}

\begin{example}
In \cite{Ryba}, minimal projective resolutions of $C_\lambda$ are described.
The minimal resolution of $C_{(4)}$ is
\[ 0 \to \mathbb{P}_{(1,1)} \to \mathbb{P}_{(2,1)} \oplus \mathbb{P}_{(3)} \to \mathbb{P}_{(4)} \to C_{(4)} \to 0.\]
It follows from eq.\ \eqref{eq:scalar_P} that $\dim \Ext^1_\F (C_{(4)},C_{(2)})=1$.
\end{example}

\begin{lemma}\label{lemma:JH}
Let $V$ a representation of degree $k$, then the composition factors of the Jordan H\"{o}lder filtration are $C_\lambda$ for $| \lambda | \leq k$ or $D_i$ for $i \leq k+1$.
\end{lemma}
\begin{proof}
It is enough to prove the claim for $V=\PP_k$ a Schur projective representation of degree $k$.
The homology functor $H_0$ is defined by $H_0(V)([n])= V/F_{n-1}V ([n])$, this functor is right exact.

Consider a Jordan H\"{o}lder filtration $G_\bigcdot$ of $\PP_k$.
Proposition 6.32 of \cite{WG14} shows that $H_0(G_i)[n]=0$ for $n>k+1$ and that $H_0(G_i)[k+1]$ is sum of sign representations of $\SG_{k+1}$.
The short exact sequence
\[  H_0(G_{i-1})[n] \to H_0(G_i)[n] \to H_0(G_i/G_{i-1})[n] \to 0\]
implies $H_0(G_i/G_{i-1})[n]=0$ for $n>k+1$ and $H_0(G_i/G_{i-1})[k+1]$ contains only the sign representation.
Now $G_i/G_{i-1}$ is simple and the homology of simple modules of degree $j$ is concentrate in degree $j$.
The equalities $H_0(C_\lambda)[\lvert \lambda \lvert]= V_\lambda$ and $H_0(D_k)[k]= V_{1^k}$ imply the result.
\end{proof}

\begin{theorem} \label{thm:gr_repr}
Let $f \colon V \to W$ be a morphism between two finitely generated $\F$-representations and $\gr_F f \colon \gr_F V \to \gr_F W$ be the corresponding graded map.
Suppose that the composition factors of the Jordan H\"{o}lder filtration of $V$ are different from $D_i$, for $i \in \N$.
Then:
\begin{enumerate}
\item $\gr_F V$ is semisimple and the addenda of $F_n/F_{n-1}$ are of the type $C_\lambda$ for $\lambda \vdash n$,
\item $\rk (f) = \rk (\gr f)$, i.e the map $f$ is strict with respect  the filtration $F_\bigcdot$.
\end{enumerate}
\end{theorem}
\begin{proof}
Consider the composition factors of $F_nV/F_{n-1}V$ have degree at least $n$ because $F_nV/F_{n-1}V([n-1])=0$ and at most $n$ by \Cref{lemma:JH} using the hypothesis that $D_{n+1}$ does not appear.
Therefore the composition factors are $C_\lambda$ for $\lambda \vdash n$.
Since $\Ext^1(C_\lambda, C_\mu)=0$ for $\lambda, \mu \vdash n$, an inductive reasoning proves that $F_n/F_{n-1}$ is semisimple.

Suppose that $f(y)=x$ with $y \in F_nV \setminus F_{n-1}V$ and $x \in F_{n-1}W$, we show that $x \in f(F_{n-1}V)$.
Let $(x)$ and $(y)$ be the subrepresentation generated by $x$ and $y$, $f$ induces a surjective morphism
\[ \overline{f} \colon \faktor{(y)}{(y) \cap F_{n-1}V} \to \faktor{(x)}{(x) \cap f(F_{n-1}V)}.\]
As noticed above $H_0((y)/ (y) \cap F_{n-1}V)$ is concentrate in degree $n$ and does not contain the sign representation, instead $H_0((x)/(x)\cap f(F_{n-1}V))[n]$ contains only some copies of the sign representation.
Thus, $H_0(\overline{f})$ is zero and surjective because $\overline{f}$ does.
This implies $H_0((x)/(x)\cap f(F_{n-1}V))=0$, so $(x)/(x)\cap f(F_{n-1}V)=0$ and hence $x \in f(F_{n-1}V)$.
\end{proof}

\subsection*{Representation theory of \texorpdfstring{$E(X)$}{E(X)}}
From now on, we assume that $\chi(X)=0$.
Let $\F(k,n)$ be the set of all maps from $[k]$ to $[n]$. 
The increasing filtration $F_\bigcdot E(X)$ defined in the previous section can be described as
\[ F_k E(X) = \sum_{f \in \F(k,n)} \im f_*.\]
The module $F_k E(X)[n]$ is, by definition, the submodule of $E(X,n)$ generated by all monomials in $G_{i,j}$ and $x_i$ (for $i,j\leq n$ and $x\in H^\bigcdot (X)$) with at most $k$ different indices.

\begin{example}
Let $x \in H^\bigcdot (X)$ be any element.
We have $G_{1,2} \in F_2 E(X) \setminus F_1 E(X)$, $x_1-x_2 \in F_1 E(X)$ and $x_1x_2 \in F_2 E(X) \setminus F_1 E(X)$.
\end{example}

We fix a graded basis $\{b_i\}_i$ of $H^\bigcdot(X)$ and assume $b_1 =1 \in H^0(X)$.
A \textit{labelled partition} $\lambda$ of $n$ is $\lambda \vdash n$ whose blocks are decorate with elements of the fixed basis.
We define $q(\lambda)=n- \size (\lambda) $ the difference between $n$ and the number of blocks $\size (\lambda) $, and $p(\lambda)=\sum_{i=1}^{\size( \lambda )} \deg (\lambda_i)$ the sum of the cohomological degrees of all the labels.
Let $L(\lambda) < \SG_n$ be a subgroup generated by disjoint cycles of length $\lambda_i$ for $i=1, \dots, \size(\lambda)$, $L(\lambda)$ is isomorphic to $\times_{i=1}^{\size (\lambda)} C_{\lambda_i}$.
Let $N(\lambda)< \SG_n$ be the subgroup of the stabilizer of $L(\lambda)$ that permutes the cycles with the same labels, it is isomorphic to $\times_{j} \SG_{n_j}$ for some $n_j$.
Finally, define $Z(\lambda)< \SG_n$ as the subgroup generated by $L(\lambda)$ and $N(\lambda)$, i.e.\ the semidirect product $L(\lambda) \rtimes N(\lambda)$.

We denote the sign representation of $\SG_i$ by $\epsilon_i$ and a faithful representation of the cyclic group $C_i$ by $\varphi_i$.
Let $\varphi_{\lambda}, \alpha_{\lambda}, \xi_{\lambda}$ be the one dimensional representations of $L(\lambda), N(\lambda), Z(\lambda)$ defined by:
\begin{align*}
&\varphi_{\lambda}= \epsilon_{n|_{L(\lambda)}} \otimes ( \boxtimes_{i=1}^{\size (\lambda)} \varphi_{\lambda_i}), \\
&\alpha_{\lambda}= \boxtimes_j \epsilon_j^{\otimes m_j}, \\
& \xi_{\lambda}= \varphi_{\lambda} \boxtimes \alpha_{\lambda},
\end{align*}
where $m_j= \lambda_i + \deg(\lambda_i)+1$ where $\lambda_i$ is any block permuted by $\SG_{n_j}$.

\begin{example}
Consider an elliptic curve $C=(S^1)^2$ and the basis of $H^\bigcdot(C)$ given by $1,x,y,xy$.
Let $\lambda=(4,4,4,4,1,1,1) \vdash 19$ with labels $(xy,xy,xy,1,x,x,x)$.
We have $\size ( \lambda )= 7$, $q(\lambda)=12$, $p(\lambda)=9$.
The associated groups are $L(\lambda) \cong (C_5)^{\times 4}$ generated by $(1,2,3,4), (5,6,7,8), (9,10,11,12), (13,14,15,16)$, $N(\lambda) \cong \SG_3 \times \SG_3$ generated by $(1,5)(2,6)(3,7)(4,8)$, $(1,9)(2,10)(3,11)(4,12)$, $(17,18)$, $ (17,19)$, and $Z(\lambda)= (C_4 \wr \SG_3) \times C_4 \times \SG_3 < \SG_{19}$.
The representations are $\varphi_\lambda = \varphi_4 \boxtimes \varphi_4 \boxtimes \varphi_4 \boxtimes \varphi_4$ (because $\epsilon_{19|_{C(\lambda)}}=1$), $\alpha= \epsilon_3 \boxtimes \epsilon_3$ (because $4+2+1$ and $1+1+1$ are odd) and $\xi_\lambda = (\varphi_4 \wr \epsilon_3) \boxtimes \varphi_4 \boxtimes \epsilon_3$.
\end{example}

A decomposition of $E(X,n)$ into $\SG_n$ representations is provided in \cite{AAB14}.
\begin{theorem}
Let $X$ be a smooth projective variety. The Kri\v{z} model decomposes as
\[ E^{p,q}(X,n) \cong \bigoplus_{\substack{q(\lambda)=q\\ p(\lambda)=p}} \Ind_{Z(\lambda)}^{\SG_n} \xi_\lambda. \]
\end{theorem}

For each representation $V=\bigoplus_\lambda V_\lambda^{n_\lambda} $ of $\SG_n$, we define the $F$-representation $C_V= \bigoplus_\lambda C_\lambda^{n_\lambda}$.
For any labelled partition $\lambda$ let $f(\lambda)$ the number of blocks of size $1$ labelled with $1 \in H^\bigcdot(X)$.

\begin{theorem}\label{thm:grad_cohomology}
Let $X$ be a smooth projective algebraic variety with $\chi(X)=0$. Then 
\[ \gr_{F}^{\bigcdot} H^{\bigcdot,\bigcdot}(E(X), \dd) =  H^{\bigcdot,\bigcdot}(\gr_{F}^{\bigcdot} E(X), \gr_{F}^{\bigcdot} \dd)\]
and for $q>0$ the $\F$-representation $E^{p,q}(X)$ has associated graded:
\[\gr_{F}^{r} E^{p,q}(X) = \bigoplus_\lambda C_{\Ind_{Z(\lambda)}^{\SG_r} \xi_\lambda},\]
where the sum is taken over all labelled partitions $\lambda \vdash r$ with $p(\lambda)=p$, $q(\lambda)=q$, and $f(\lambda)=0$.
\end{theorem}
\begin{proof}
We first prove that the sign representation $\epsilon_n$ does not appear in $E^{p,q}(X,n)$ for $q>0$.
We have
\begin{align*}
\langle \epsilon_n, E^{p,q}(X,n) \rangle_{\SG_n} &= \sum_{\lambda} \langle \epsilon_n, \Ind_{Z(\lambda)}^{\SG_n} \xi_\lambda \rangle_{\SG_n} \\
&=\sum_{\lambda} \langle \epsilon_{n|_{Z(\lambda)}}, \xi_\lambda \rangle_{Z(\lambda)} \\
&=\sum_{\lambda} \langle \epsilon_{n|_{L(\lambda)}}, \varphi_\lambda \rangle_{L(\lambda)} \langle \epsilon_{n|_{N(\lambda)}}, \alpha_\lambda \rangle_{N(\lambda)} \\
&=\sum_{\lambda} \langle 1_{L(\lambda)}, \boxtimes_{i=1}^{\|\lambda \|} \varphi_{\lambda_i} \rangle_{L(\lambda)} \langle \epsilon_{n|_{N(\lambda)}}, \alpha_\lambda \rangle_{N(\lambda)} \\
&=\sum_{\lambda} \langle \epsilon_{n|_{N(\lambda)}}, \alpha_\lambda \rangle_{N(\lambda)}  \prod_{i=1}^{\| \lambda \|}\langle 1_{\lambda_i}, \varphi_{\lambda_i} \rangle_{C_{\lambda_i}} =0,
\end{align*}
because $\langle 1_{\lambda_i}, \varphi_{\lambda_i} \rangle_{C_{\lambda_i}}\neq 0$ if and only if  $1_{\lambda_i}=\varphi_{\lambda_i}$ if and only if $\lambda_i=1$.
The sign representation appears only in $E^{p,0}(X,n)$.

It follows that $D_i$ cannot be a composition factor of $E^{p,q}(X)$ for $q>0$, so by \Cref{thm:gr_repr} we obtain that $\dd$ is strict. 

The second part of the statement follows from the fact that the submodule $\Ind_{Z(\lambda)}^{\SG_n} \xi_\lambda$ is contained in $F_rE^{p,q}(X)([n])$ if and only if $f(\lambda)\geq n-r$.
Therefore,
\begin{equation*}
\faktor{F_rE^{p,q}(X)}{F_{r-1}E^{p,q}(X)}([r]) = \bigoplus_{\lambda \textnormal{ s.t.\ } f(\lambda)=0} \Ind_{Z(\lambda)}^{\SG_n} \xi_\lambda,
\end{equation*}
and since $\gr^r_F E^{p,q}(X)$ is semisimple, we obtain the claimed equality.
\end{proof}

\begin{example}
In the case $X=C$, we have $\dd(G_{i,j})=(x_i-x_j)(y_i-y_j)$ and 
$\gr_F \dd (G_{i,j})= -x_iy_j-x_jy_i$.
Moreover $\gr_F \dd (x_i G_{i,j}) = x_ix_jy_j - x_ix_jy_i $ and notice that the product of $x_i$ and $G_{i,j}$ in $\gr^\bigcdot _F E(C)$ is zero.
\end{example}

\begin{remark}
Let $M$ be a even-dimensional smooth manifold with $\chi(M)=0$.
The analogous result of \Cref{thm:grad_cohomology} holds for the Leray spectral sequence for the inclusion $\conf(M,n)\hookrightarrow M^n$.
Indeed the collection of spectral sequences for each $n\in \N$ is an $\F$-module and the filtration $F$ is strictly compatible with the differentials $\dd_i$ (for all $i \in \N$) because the sign representation can appears only in the $0$-th row (i.e.\ $q=0$).
\end{remark}
\subsection*{Behaviour of mixed Hodge numbers}

Recall that every polynomial $P(t)$ in $\Q[t]$ such that $P(n) \in \N$ for all $n \in \N$ can be written uniquely as
\[ P(t) = \sum_{i=0}^{\deg P} a_i \binom{t}{i},\]
for some $a_i \in \Z$.
The value $a_i$ not need to be positive.

\begin{example}
Consider $\dim H^2(\conf(S^3,n);\Q)$ for $n>0$, it is known that is polynomial in $n$ and it value is $\frac{(n-1)(n-2)}{2}$.
However $a_1=-1$ because  
\[ \dim H^2(\conf(S^3,n);\Q) = \frac{(n-1)(n-2)}{2} = \binom{n}{2} - \binom{n}{1}+ \binom{n}{0}.\]
Indeed $H^2(\conf(S^3)) \cong D_3$ as $\F$-module (included in position $(0,2)$ of the Leray spectral sequence).
\end{example}

We collect the information about $\dim H^{p,q}(E(X,n),\dd)$ in a polynomial 
\begin{equation*}
P^{p,q}(n):= \dim H^{p,q}(E(X,n),\dd) = \sum_{i=0}^{p+2q} a^{p,q}_i \binom{n}{i},
\end{equation*}
for some unique integers $a^{p,q}_i$.
\begin{proposition}\label{prop:coeff_binom}
For $q>0$, the coefficients $a_i^{p,q}$ are positive and coincide with
\[ a_i^{p,q} = \dim H^{p,q} \left( \faktor{E(X,i)}{F_{i-1}E(X,i)},\dd \right).\]
\end{proposition}

\begin{proof}
We use \Cref{thm:grad_cohomology} and the fact that $\gr^i_F E^{p,q}(X)$ has not composition factors of type $D_k$ to obtain:
\begin{align*}
\dim H^{p,q}(E(X,n),\dd) &= \sum_{i\leq p+2q} \dim \gr^i_F H^{p,q}(E(X,n)) \\
&=\sum_{i\leq p+2q} \dim H^{p,q}(\gr^i_F E(X,n)) \\
&= \sum_{i\leq p+2q} \dim H^{p,q} \left( \Ind_{\SG_i \times \SG_{n-i}}^{\SG_n} \faktor{E(X,i)}{F_{i-1}E(X,i)} \right)\\
&= \sum_{i\leq p+2q} \dim \Ind_{\SG_i \times \SG_{n-i}}^{\SG_n}  H^{p,q}\left( \faktor{E(X,i)}{F_{i-1}E(X,i)} \right)\\
&= \sum_{i\leq p+2q} \binom{n}{i} \dim H^{p,q} \left( \faktor{E(X,i)}{F_{i-1}E(X,i)} \right)
\end{align*}
Since $\N$ is infinite we obtain the corresponding equality between polynomials.
\end{proof}

\section{The elliptic case}

Let $C$ be an elliptic curve, topologically $C=(S^1)^2$, the cohomology $H^\bigcdot(C;\Q)$ is the exterior algebra on two generator $x$ and $y$ such that $x \smile y = [C]$.
The construction of the previous section is compatible with the action of $\SL$.

We recall two results from \cite[Lemma 1.6, Theorem 3.9]{Pag}.
\begin{proposition}\label{prop:old}
The cohomology $H^k(\conf(C,n))$ vanishes for $k>n+1$ and so $a_i^{p,q}=0$ for $p+q>i+1$ or $p+2q<i$.
Moreover the mixed Hodge numbers for $q=0$ are given by:
\[P^{p,0}(t)=(p+1) \binom{t}{p} + (p-1) \binom{t}{p-1}.\]
\end{proposition}
Indeed, an easy computation shows that 
\[H^{p,0}(\conf(C))= (\mathbb{P}_{1^p} \boxtimes \mathbb{V}_p) \oplus (\mathbb{P}_{1^{p-1}} \boxtimes \mathbb{V}_{p-2}),\]
as representation of $\F \times \SL$.

\begin{remark}
The multiplicity of the representation $C_{(k)}$ for $k>1$ in the graded cohomology $\gr_F^\bigcdot H^\bigcdot(\conf(C))$ can be deduced from \cite{PagUnorder}.
\end{remark}

For the sake of notation, we denote $\Ind_{\SG_n \times \SG_m}^{\SG_{n+m}} V \boxtimes W$ by $V \cdot W$.


\begin{lemma}
The following decomposition of $\SG_{p+2q}\times \SL$-representations holds:
\[\gr^{p+2q}_F E^{p,q}(C,p+2q) \cong \bigoplus_{a=0}^{\lfloor \frac{p}{2}\rfloor} \left( \Ind_{\SG_2 \wr \SG_q}^{\SG_{2n}} V_{(2)}^{\boxtimes q} \otimes V_{(1^q)} \right)  \cdot V_{(2^a,1^k)} \boxtimes \mathbb{V}_k, \]
where $p=2a+k$.
\end{lemma}
\begin{proof}
From \Cref{thm:grad_cohomology}, we have 
\[ \gr^{p+2q}_F E^{p,q}(C,p+2q) \cong \bigoplus_{b=0}^{p} \Bigl( \Ind_{C_2 \wr \SG_q}^{\SG_{2q}} V_{(2)}^{\boxtimes q} \otimes V_{(1^q)} \Bigr) \cdot V_{(1^{p-b})} \cdot V_{(1^{b})} \]
because the labelled partitions $\lambda\vdash p+2q$ with $p(\lambda)=p$, $q(\lambda)=q$ and $f(\lambda)=0$ are $(2^q,1^{p-b},1^b)$ for some $b$ with blocks labelled respectively by $1$, $x$, and $y$.
Moreover for such $\lambda$ we have $\varphi_\lambda = V_{(2)}^{\boxtimes q}$ and $\alpha_\lambda = V_{(1^q)} \boxtimes V_{(1^{p-b})} \boxtimes V_{(1^{b})}$.
By definition the maximal torus of $\SL$ acts with weight $p-2b$ on the $b$-th addendum, so 
\[ \gr^{p+2q}_F E^{p,q}(C,p+2q) \cong \bigoplus_{a=0}^{\lfloor \frac{p}{2}\rfloor} W \cdot \Bigl(V_{(1^{a+k})} \cdot V_{(1^{a})} \ominus V_{(1^{a+k+1})} \cdot V_{(1^{a-1})} \Bigr) \boxtimes \mathbb{V}_{k}, \]
where $k=p-2a$ and $W=\Ind_{C_2 \wr \SG_q}^{\SG_{2q}} V_{(2)}^{\boxtimes q} \otimes V_{(1^q)}$.
The representation $V_{(1^{a+k})} \cdot V_{(1^{a})} \ominus V_{(1^{a+k+1})}\cdot V_{(1^{a-1})}$ has dimension $\binom{p}{a}-\binom{p}{a-1}$.
Using the Littlewood-Richardson rule (see \cite{Fulton}), we observe that the representation $V_{(2^a,1^k)}$ appears in $V_{(1^{a+k})} \cdot V_{(1^{a})}$ but not in $V_{(1^{a+k+1})}\cdot V_{(1^{a-1})}$.
The hook formula shows that $\dim V_{(2^a,1^k)} = \frac{p!(k+1)}{a! a+k+1!} = \frac{k+1}{a+k+1}\binom{p}{a}$ and an easy computation show that 
\[\dim V_{(2^a,1^k)} = \dim \left( V_{(1^{a+k})} \cdot V_{(1^{a})} \ominus V_{(1^{a+k+1})}\cdot V_{(1^{a-1})} \right).\]
Since the first representation is contained in the second one, we complete the proof.
\end{proof}

\subsection*{Oyster partitions}
We need the Frobenius notation for partitions: for any sequences $a_1> a_2 > \dots  > a_k  > 0$ and $b_1> b_2 > \dots > b_k > 0$ let $(a_1, \dots, a_k \mid b_1, \dots, b_k)$ be the partition of $\sum_{i=1}^k (a_i+b_i -1)$ such that the $i$-th row has length $a_i+i-1$ and the $i$-th column has length $b_i+i-1$ for $i\leq k$
Let $Q(n)$ be the set of all the partitions of $n$ of the form $(a_1, \dots, a_k \mid a_1-1, \dots, a_k-1)$.
From \cite[Proposition 2.3.9 (a)]{Weyman}, \cite[Theorem A2.8]{Stanley}, or \cite[Appendix A eq.\ (6.2)]{Macdonald}, we have
\[ \Ind_{\SG_2 \wr \SG_q}^{\SG_{2n}} V_{(2)}^{\boxtimes q} \otimes V_{(1^q)} \cong \bigoplus_{\lambda \in Q(2q)} V_{\lambda}. \]

\begin{definition}
A $k$-\textit{core partition} of $2q+k$ is any partition of the form $(\lambda_1, \dots, \lambda_k)$ such that $(\lambda_1-1, \dots, \lambda_k-1)$ is in $Q(2q)$.
A $(k,a)$-\textit{shell partition} is any partition of the form $(b_1+3, \dots, b_a+3 \mid b_1, \dots, b_a)$ with $b_a>k$.
A $(k,a)$-\textit{oyster partition} is a partition $(c_1, \dots, c_{a+k} \mid d_1, \dots, d_{a+k})$ such that $(c_1, \dots, c_a \mid d_1, \dots, d_{a})$ is a $(k,a)$-shell partition and $(c_{a+1}, \dots, c_{a+k} \mid d_{a+1}, \dots, d_{a+k})$ is a $k$-core partition.
\end{definition}

\begin{example}
Let $k=2$ and consider the $2$-core partition $(4,4)$ obtained from $(3,2 \mid 2,1) \in Q(6)$ and the $(2,1)$-shell partition $(6,1,1)$.
The union of the shell and the core give the following oyster partition $\lambda=(6,5,5)$:
\[\ytableaushort
  {\none \none \none \none 1 1,\none \none \none \none 2, \none \none \none \none 3}
 * {5,5,5}
 * [*(yellow)]{6,1,1}
\]
The representation $V_\lambda \boxtimes \mathbb{V}_2$ is an addendum of 
\[V_{(4,4,4)}\cdot V_{(2,1,1)} \boxtimes \mathbb{V}_2 \subset \left(\Ind_{\SG_2 \wr \SG_6}^{\SG_{12}} V_{(2)}^{\boxtimes 6} \otimes V_{(1^6)} \right) \cdot V_{(2,1,1)} \boxtimes \mathbb{V}_2 \subset \gr^{16}_F E^{4,6}(C,16).\]
Indeed, $6=3+2+1$ and the partition $(4,4,4)$ is in Frobenius notation $(4,3,2 \mid 3,2,1)$.
The Littlewood-Richardson rule shows that the multiplicity of $V_\lambda$ in $V_{(4,4,4)}\cdot V_{(2,1,1)}$ is one and correspond to the skew semistandar Young tableaux of shape $(6,5,5)/ (4,4,4)$ of content $(2,1,1)$ shown above.
\end{example}

\begin{lemma}\label{lemma:oyster}
The module $\gr^{p+2q}_F H^{p,q}(\conf(C))$ contains
\[ \bigoplus_{a=0}^{\lfloor \frac{p}{2} \rfloor} \bigoplus_{\lambda \, (k,a)\textnormal{-oyster}} C_{V_\lambda} \boxtimes \mathbb{V}_k, \]
where $k=p-2a$ and the sum is taken over all $(k,a)$-oyster partitions of $p+2q$.
\end{lemma}
\begin{proof}
It is enough to prove that for every $(k,a)$-oyster partition $\lambda \vdash p+2q$ the representation $V_\lambda \boxtimes \mathbb{V}_k$ appears in $\gr^{p+2q}_F E^{p,q}(C,p+2q)$ but not in $\gr^{p+2q}_F E^{p+2,q-1}(C,p+2q)$ neither in $\gr^{p+2q}_F E^{p-2,q+1}(C,p+2q)$.

For $\lambda= (\lambda_1, \dots, \lambda_r)$ consider the partition $\mu = (\mu_1, \dots, \mu_r)$ such that $\mu_i= \lambda_i-2$ for $i\leq a$, $\mu_i= \lambda_i-1$ for $a<i\leq a+k$, and $\mu_i= \lambda_i$ for $a+k<i$.
By construction $\mu$ belongs to $Q(2q)$ and so $V_\mu$ appears in $V_{(1^q)}[V_{2}]$.
Applying the Littlewood-Richardson rule to $V_\mu \cdot V_{(2^a,1^k)}$, we obtain that $V_\lambda$ has multiplicity one in $V_\mu \cdot V_{(2^a,1^k)}$.
We have proven  $V_\lambda \boxtimes \mathbb{V}_k \subset \gr^{p+2q}_F E^{p,q}(C)$.

In order to show that $\gr^{p+2q}_F E^{p-2,q+1}(C,p+2q)$ does not contain $V_\lambda \boxtimes \mathbb{V}_k$ is equivalent to show that $V_\lambda$ is not contained in $V_{\eta} \cdot V_{(2^{a-1},1^k)}$ for all $\eta \in Q(2q+2)$.
Suppose by contradiction that there exist $\eta \in Q(2q+2)$ and a Littlewood-Richardson skew tableau of shape $\lambda / \eta$ and content $(2^{a-1},1^k)$.
The $i$-th column of $\lambda$ has length $\lambda_i-3$ for $i\leq a$ and the one of $\mu$ has length $\mu_i-1$ for all $i$.
Since $\mu$ is contained in $\lambda$ then $\mu_i \leq \lambda_i -2$ for all $i\leq a$.
In particular $\lambda / \eta$ has at least $2a$ boxes in the first $a$ rows.
From the reverse lattice word property of $\lambda / \eta$ the number $j$ cannot appear in the first $j-1$ rows, hence $\lambda / \eta$ has at most $2a-1$ boxes in the first $a$ rows.
We have obtained a contradiction.

Suppose that $V_\lambda$ is contained in $V_{\eta} \cdot V_{(2^{a+1},1^k)}$ for some $\eta \in Q(2q-2)$.
The constrictions of $\eta \in Q(2q-2)$ and the existence of a Littlewood-Richardson skew tableau of shape $\lambda / \eta$ and content $(2^{a+1},1^k)$ imply that $\eta_i=\lambda_i-2$ for all $i\leq a$.
The inequality $\eta_i\leq \lambda_i-2$ holds because columns of $\eta$ are shorter than the corresponding ones of $\lambda$, the other inequality $\eta_i\geq \lambda_i-2$ holds because the skew tableaux $\lambda / \eta$ has at most $2i$ elements in the first $i$-th rows for all $i$.
Therefore $\lambda / \eta$ has at most $a+k$ nonempty rows, but each Littlewood-Richardson skew tableau of content $(2^{a+1},1^k)$ must have at least $a+k+1$ nonempty rows.
This is contrary to the hypothesis that $V_\lambda$ is contained in $V_{\eta} \cdot V_{(2^a+1,1^k)}$ for some $\eta \in Q(2q-2)$.
\end{proof}

\begin{corollary}\label{cor:oyster}
We have the following lower bounds for the bi-graded Betti numbers:
\begin{align*}
\gr_F^{2q+2} H^{2,q} (E(C), \dd) & \supseteq C_{(q+3 \mid q)}  \boxtimes \mathbb{V}_0, \\
\gr_F^{8} H^{2,3} (E(C),\dd) &= C_{(6 \mid 3)}  \boxtimes \mathbb{V}_0 \oplus C_{(4,3 \mid 2,1)}  \boxtimes \mathbb{V}_2, \\
\gr_F^{6} H^{2,2} (E(C),\dd) &= C_{(5 \mid 2)}  \boxtimes \mathbb{V}_0 \oplus C_{(4,1 \mid 2,1)}  \boxtimes \mathbb{V}_2, \\
\gr_F^{4} H^{2,1}(E(C),\dd) &= C_{(4 \mid 1)}  \boxtimes \mathbb{V}_0 \oplus C_{(3 \mid 2)}  \boxtimes \mathbb{V}_2, \\
\gr_F^{3} H^{1,1}(E(C),\dd) &= C_{(3 \mid 1)} \boxtimes \mathbb{V}_1.
\end{align*}
In particular $\dim \gr_F^{2q+2} H^{2,q} (E(C,n), \dd)\geq \binom{2q+1}{q-1} \binom{n}{2q+2}$.
\end{corollary}
\begin{proof}

The first inclusion follows from \Cref{lemma:oyster} and the fact that $(q+3|q)$ is an $(0,2)$-oyster partition of $2q+2$ for $q>0$.
The following oyster partitions (with empty shells):
\[\ytableaushort
  {\none \none \none 1, \none \none \none 2}
 * {4,4}
 \qquad
 \ytableaushort
  {\none \none \none 1, \none 2}
 * {4,2}
  \qquad
 \ytableaushort
  {\none \none 1, 2}
 * {3,1}
   \qquad
 \ytableaushort
 {\none \none 1}
 * {3}
\]
imply that the right hand sides are contained in the left hand sides.
The other containment follows from a dimensional argument:
the dimensions of the left hand sides are computed in \Cref{table:e_3_prim,table:e_4_prim,table:e_6_prim} and eq.\ \eqref{eq:dim_H23_8} and coincide with the dimension of the representations on the right.
Finally, since $\dim V_{(q+3,1^{q-1})}= \binom{2q+1}{q-1}$, it follows that $\dim C_{V_{(q+3,1^{q-1})}}([n]) \boxtimes \mathbb{V}_0= \binom{2q+1}{q-1}\binom{n}{2q+2}$.
\end{proof}

%

\subsection*{Upper bounds for Betti numbers}

We denote by $E^{p,q}(C)_k$ the subspace in $ E^{p,q}(C)$ of highest vectors for $\SL$ of weight $k$ and similar for $\gr_F E^{p,q}(C)_k$ and for $H^{p,q}(E(C),\dd)_k$.
Therefore
\[  E^{p,q}(C) \cong \bigoplus_{a=0}^{\lfloor \frac{p}{2} \rfloor}  E^{p,q}(C)_{p-2a} \boxtimes \mathbb{V}_{p-2a}. \]
In order to give upper bounds for Betti numbers, we need the following result.
\begin{lemma}
\label{lemma:vanishing}
The following cohomology groups are zero:
\begin{align}
& H^{0,q}(E(C),\dd)=0 && q>0, \label{eq:zero_col_zero}\\
& \gr_F^{2q+1} H^{1,q}(E(C),\dd)=0 && q>1, \label{eq:zero_col_one}\\
& \gr_F^{2q} H^{1,q}(E(C),\dd)=0 && q>2, \label{eq:zero_col_one_up}\\
& \gr_F^{2q+2} H^{2,q}(E(C),\dd)_2=0 && q>3. \label{eq:zero_col_two}
\end{align}
\end{lemma}
\begin{proof}
Equation \eqref{eq:zero_col_zero} is proven in \cite[Proposition 1.2]{AAB14}.

For eq.\ \eqref{eq:zero_col_one} we proceed by induction, the base case follows from the entry $(1,2)$ of \Cref{table:e_5_prim}.
Let $G_\bigcdot$ the filtration of the complex $D_r = \oplus_q \gr^r_{F} E^{r-2q,q}(C)$ for fixed $r$ defined by $G_0=0$, $G_2=D_r$ and
\[ G_1^q = \langle x_1 \alpha, y_1 \alpha \mid \alpha \in E^{r-1-2 q, q}(C,\{2, \dots, r\}) \rangle. \]
The complex $G_1$ is isomorphic to two copies of  $\gr_F^{r-1} E^{r-1-2q, q}(C)$ and the quotient complex $G_2/G_1$ is identified with $2q$ copies of $\gr_F^{r-2} E^{r-2q, q-1}(C)$ (one for each $G_{1i}$).
For $r=2q+1$ we have:
\begin{align*}
\dim \gr_F^{2q+1} &H^{1,q}(E(C),\dd) \leq \dim H^{q}(G_1) + \dim H^{q}(G_2/G_1) \\
&= 2\dim \gr_F^{2q} H^{0,q}(E(C),\dd) + 2q \dim \gr_F^{2q-1} H^{1,q-1}(E(C),\dd).
\end{align*}
The first addendum is zero by eq.\ \eqref{eq:zero_col_zero} and the second one by inductive step. 

For eq.\ \eqref{eq:zero_col_one_up} we proceed by induction, the base case follows from the entry $(1,3)$ of \Cref{table:e_6_prim}.
Consider the filtration $G'_\bigcdot$ of $D'_r = \oplus_q \gr^r_F E^{r-2q+1,q}(C)$ defined by $G'_0=0$, $G'_3=D'_r$,
\begin{align*}
G'_1 &= \langle x_1 \alpha, y_1 \alpha, x_1y_1 \beta \mid \alpha, \beta \textnormal{ w/o index } 1 \rangle, \\
G'_2 &= G_1' + \langle G_{1,i}\alpha, G_{1,i}x_1 \beta, G_{1,i}y_1 \beta, G_{1,i}x_1y_1 \gamma \mid \alpha, \beta,\gamma \textnormal{ w/o indices } 1,i  \rangle_{i=2,\dots,r}.
\end{align*}
For $r=2q$, we have 
\begin{equation}
\label{eq:G1}
H^{q}(G'_1) \cong \gr_F^{2q-1} H^{0,q}(E(C)) \oplus y_1 \gr_F^{2q-1} H^{0,q}(E(C))=0
\end{equation}
by eq.\ \eqref{eq:zero_col_zero}.
Similarly, $H^{q}\left( G'_2/G_1' \right)$ is equal to
\begin{equation}
\label{eq:G2}
\gr_F^{2q-2} H^{1,q-1}(E(C))^{\oplus 2q-1} \oplus \gr_F^{2q-2} H^{0,q-1}(E(C))^{\oplus 4q-2} =0
\end{equation}
by inductive step and by eq.\ \eqref{eq:zero_col_zero}.
The top cohomology vanishes:
\begin{equation}
\label{eq:G3}
H^{q}\left( G'_3/G_2' \right) \cong \bigoplus_{1<i< j} G_{1,i}G_{1,j} \gr_F^{2q-3} H^{1,q-2}(E(C)) =0
\end{equation}
because of eq.\ \eqref{eq:zero_col_one}.
Putting together eq.\ \eqref{eq:G1}, \eqref{eq:G2} and \eqref{eq:G3}, we obtain the claimed equality $\gr_F^{2q}H^{1,q}(E(C),\dd)=0$.

For eq.\ \eqref{eq:zero_col_two} we proceed by induction, the base case $\gr_F^{10} H^{2,4}(E(C),\dd)_2=0$ is computed with the Python code available at \url{https://www.dm.unibo.it/~roberto.pagaria/Top_graded_cohom_order_config_elliptic_curve.py}.
The computation involves only the $2$-weight space for $T\subset \SL$, i.e.\ the subspace of the graded module $\gr^r_F E^{r-2q,q}(C,r)$ generated by the monomials with $2$ more occurrences of $x$ than $y$.
Consider the filtration $G_\bigcdot''$ defined as the filtration $G_\bigcdot$ restricted to the subspace of weight $2$ for the torus action.
We need to prove for $r=2q+2$ that $H^q(G_1'')=0$ and $H^q(G_2''/G_1'')=0$.
The first equality follows from $H^q(G_1'')=\gr_F^{2q+1} H^{1,q}(E(C),\dd)_1=0$ by eq.\ \eqref{eq:zero_col_one}.
The complex $G_2''/G_1''$ is a direct sum of $2q+1$ copies of the complex $\gr^{2q}_F E(C)_2$ and so $H^q(G_2''/G_1'')\cong \gr_F^{2q+1} H^{2,q}(E(C),\dd)_2^{\oplus 2q+1}=0$ by inductive step.
\end{proof}

The below Corollary follows from the previous Theorem and the numerical computation exposes in the next section.
\begin{corollary}
\label{cor:Betti}
The Betti numbers of $\conf(C,n)$ are:
\begin{align*}
&b_0 =1, \\ 
&b_1=2n, \\
&b_2= 2\binom{n}{3}+ 3\binom{n}{2}+ n, \\
&b_3=  14 \binom{n}{4} + 8 \binom{n}{3} + 2 \binom{n}{2}, \\
&b_4 = 32 \binom{n}{6} + 74 \binom{n}{5}+ 32 \binom{n}{4} + 5 \binom{n}{3}, \\
&b_5 = 63 \binom{n}{8} + 427 \binom{n}{7} + 490 \binom{n}{6} + 154 \binom{n}{5}+18 \binom{n}{4}, \\
&b_k = c_k \binom{n}{2k-2} + o(n^{2k-2}),
\end{align*}
where $c_k \geq \binom{2k-3}{k-3}$.
\end{corollary}
\begin{proof}
The case $b_i$ for $i \leq 5$ are computed from \Cref{table:e_3_prim,table:e_4_prim,table:e_5_prim,table:e_6_prim,table:e_7_prim}.
Recall that the Poincaré polynomial of $\conf(C,n)$ and of $\conf(C,n)/C$ differ by a factor $(1+t)^2$ and the case $q=0$ follows from \Cref{prop:old}.
The case $b_5$ need the vanishing results of \Cref{lemma:vanishing} and eq.\ \eqref{eq:dim_H23_8}.

For general $k$, we have
\[ b_k(n)= \sum_{p+q=k} \sum_{i=p+q-1}^{p+2q} \binom{n}{i} \dim H^{p,q} \left( \faktor{E(C,i)}{F_{i-1}E(C,i)} \right).\]
Eq.\ \eqref{eq:zero_col_zero} and \eqref{eq:zero_col_one} ensure that the polynomial has degree at most $2k-2$.
Eq.\ \eqref{eq:zero_col_one_up} and \eqref{eq:zero_col_two} implies that 
\[ b_k(n)= \binom{n}{2k-2} \dim H^{2,k-1} \left( \faktor{E(C,2k-2)}{F_{2k-3}E(C,2k-2)} \right)_0 + o(n^{2k-2}).\]
Finally, \Cref{cor:oyster} implies the desired result.
\end{proof}

\begin{conjecture}
We claim that
\[ b_k = \binom{2k-3}{k-3} \binom{n}{2k-2} + o(n^{2k-2}).\]
\end{conjecture}

\begin{remark}
Let $S$ be a connected orientable surface of finite type, the stable range for $H^k(\conf(S,n);\Q)$ in the sense of \cite{CF13} is $n\geq 4k$ as proven in \cite[Theorem 1]{Church12}.
The above discussion implies that for $k>2$ the stable range in the elliptic case $S=C$ is $n\geq 4k-4$ and that this bound is optimal.
\end{remark}

\appendix
\section{Small cases}
\label{appendix}
\input{Appendix}

\subsection*{Acknowledgement}
I would thank Gian Marco Pezzoli for the useful discussions and John Wiltshire-Gordon for notifying me the reference \cite{Ryba}.

\bibliographystyle{amsalpha}
\bibliography{orderedconf}

\end{document}

%% file: Appendix.tex
The elliptic curve $C$ acts on $\conf(C,n)$ by $p \cdot (p_1, \dots, p_n)= (p_1+p, \dots, p_n+p)$ where $+$ is the group operation on the elliptic curve $C$.
This action is compatible with the structure of $\F$-module and the fibration 
\[ C \to \conf(C,n) \to \faktor{\conf(C,n)}{C}\]
has a non-canonical section $s \colon \faktor{\conf(C,n)}{C} \to \conf(C,n)$.
This induces an isomorphism 
\begin{equation}\label{eq:split_C}
H^\bigcdot (\conf(C,n)) \cong H^\bigcdot (C) \otimes H^\bigcdot (\conf(C,n)/C)
\end{equation}
as rings, but not as $\F$-modules because the section cannot be chosen in an equivariant way.
We used a Python3 to compute the cohomology of $\conf(C,n)/C$ for $n\leq 7$.
The code is available at \url{https://www.dm.unibo.it/~roberto.pagaria/Cohom_order_config_elliptic_curve.py}. and the bigraded Betti numbers are presented in \Cref{table:e_2,table:e_3,table:e_4,table:e_5,table:e_6,table:e_7}.

Since $\gr^r_F H^{p,q}(\conf(C))$ is semisimple for $q>0$, we have a decomposition analogous to eq.\ \eqref{eq:split_C}:
\[\gr_F^\bigcdot H^{\bigcdot,q} (\conf(C,n)) \cong \gr_F^\bigcdot H^\bigcdot (C) \otimes \gr_F^\bigcdot H^{\bigcdot,q} (\conf(C,n)/C) \]
\Cref{table:e_3_prim,table:e_4_prim,table:e_5_prim,table:e_6_prim,table:e_7_prim} report the numbers $a_i^{p,q}$ for $q>0$ associated to the $\F$-module $\gr_F^\bigcdot H^{p,q} (\conf(C,n)/C)$.
These entries are computed from the corresponding ones of \Cref{table:e_2,table:e_3,table:e_4,table:e_5,table:e_6,table:e_7} as convolution with binomial coefficients.

As an example we consider $\dim H^{2,2}(\conf(C,n)/C)$, the values in tables  \Cref{table:e_2,table:e_3,table:e_4,table:e_5,table:e_6,table:e_7} are $0,0,0,38,260,1022$ and coincides with the evaluation of the polynomial
\[ 38 \binom{n}{5} + 32 \binom{n}{6}\]
at $n=2,3,4,5,6,7$. The coefficient of this polynomial are reported in the corresponding entries $(2,2)$ of \Cref{table:e_3_prim,table:e_4_prim,table:e_5_prim,table:e_6_prim,table:e_7_prim}.
(Notice that in those tables the $0$-th row is omitted.)

The values of \Cref{table:e_2,table:e_3} corresponding to the cases $n=2,3$ were previously computed in \cite{Azam}.

\begin{table}
\centering
\begin{tabular}{cc}
$0$ & \\
$1$ & $2$ 
\end{tabular}
\caption{The dimension of the cohomology $H^{p,q}(\conf(C,2)/C)$.}
\label{table:e_2}
\end{table}

\begin{table}
\centering
\begin{tabular}{ccc}
$0$ & &  \\
$0$ & $2$ &  \\
$1$ & $4$ & $3$ 
\end{tabular}
\caption{The dimension of the cohomology $H^{p,q}(\conf(C,3)/C)$.}
\label{table:e_3}
\end{table}

\begin{table}
\centering
\begin{tabular}{cccc}
$0$ & & & \\
$0$ & $4$ & &  \\
$0$ & $8$ & $10$ &  \\
$1$ & $6$ & $9$ & $4$ 
\end{tabular}
\caption{The dimension of the cohomology $H^{p,q}(\conf(C,4)/C)$.}
\label{table:e_4}
\end{table}

\begin{table}
\centering
\begin{tabular}{ccccc}
$0$ & & & & \\
$0$ & $12$ & & & \\
$0$ & $20$ & $38$ & &  \\
$0$ & $20$ & $50$ & $24$ &  \\
$1$ & $8$ & $18$ & $16$ & $5$ 
\end{tabular}
\caption{The dimension of the cohomology $H^{p,q}(\conf(C,5)/C)$.}
\label{table:e_5}
\end{table}

\begin{table}
\centering
\begin{tabular}{cccccc}
$0$ & & & & & \\
$0$ & $48$ & & & & \\
$0$ & $72$ & $176$ & & & \\
$0$ & $60$ & $260$ & $152$ & & \\
$0$ & $40$ & $150$ & $144$ & $50$ & \\
$1$ & $10$ & $30$ & $40$ & $25$ & $6$ 
\end{tabular}
\caption{The dimension of the cohomology $H^{p,q}(\conf(C,6)/C)$.}
\label{table:e_6}
\end{table}

\begin{table}
\centering
\begin{tabular}{ccccccc}
$0$ & & & & & & \\
$0$ & $240$ & & & & & \\
$0$ & $336$ & $976$ & & & & \\
$0$ & $252$ & $1491$ & $1040$ & & & \\
$0$ & $140$ & $1022$ & $1232$ & $425$ & & \\
$0$ & $70$ & $350$ & $504$ & $350$ & $90$ & \\
$1$ & $12$ & $45$ & $80$ & $75$ & $36$ & $7$
\end{tabular}
\caption{The dimension of the cohomology $H^{p,q}(\conf(C,7)/C)$.}
\label{table:e_7}
\end{table}

\begin{table}
\centering
\begin{tabular}{ccc}
$0$ & \\
$0$ & $2$ 
\end{tabular}
\caption{
The coefficients $a_3^{p,q}$ of the $\F$-module $H^{p,q}(\conf(C)/C)$ ($q>0$).}
\label{table:e_3_prim}
\end{table}

\begin{table}
\centering
\begin{tabular}{cccc}
$0$ & & \\
$0$ & $4$ & \\
$0$ & $0$ & $10$ 
\end{tabular}
\caption{The coefficients $a_4^{p,q}$ of the $\F$-module $H^{p,q}(\conf(C)/C)$ ($q>0$).}
\label{table:e_4_prim}
\end{table}

\begin{table}
\centering
\begin{tabular}{ccccc}
$0$ & & &  \\
$0$ & $12$ & &  \\
$0$ & $0$ & $38$ &   \\
$0$ & $0$ & $0$ & $24$ 
\end{tabular}
\caption{The coefficients $a_5^{p,q}$ of the $\F$-module $H^{p,q}(\conf(C)/C)$ ($q>0$).}
\label{table:e_5_prim}
\end{table}

\begin{table}
\centering
\begin{tabular}{cccccc}
$0$ & & & & \\
$0$ & $48$ & & & \\
$0$ & $0$ & $176$ & & \\
$0$ & $0$ & $32$ & $152$ & \\
$0$ & $0$ & $0$ & $0$ & $50$ 
\end{tabular}
\caption{The coefficients $a_6^{p,q}$ of the $\F$-module $H^{p,q}(\conf(C)/C)$ ($q>0$).}
\label{table:e_6_prim}
\end{table}

\begin{table}
\centering
\begin{tabular}{ccccccc}
$0$ & & & & & \\
$0$ & $240$ & & & & \\
$0$ & $0$ & $976$ & & & \\
$0$ & $0$ & $259$ & $1040$ & & \\
$0$ & $0$ & $0$ & $168$ & $425$ & \\
$0$ & $0$ & $0$ & $0$ & $0$ & $90$ 
\end{tabular}
\caption{The coefficients $a_7^{p,q}$ of the $\F$-module $H^{p,q}(\conf(C)/C)$ ($q>0$).}
\label{table:e_7_prim}
\end{table}

Moreover, using the same code we have:
\begin{equation}\label{eq:dim_H23_8}
H^{2,3}(\conf(C,8)/C;\Q) = 7063,
\end{equation}
and so $a^{2,3}_8=63$.